\documentclass[11pt,a4paper]{smfart} 
\usepackage{mathsmf}
\usepackage{smfthm}
\usepackage{graphicx}
\usepackage{color}
\usepackage[utf8]{inputenc}
\usepackage[T1]{fontenc}
\usepackage[french]{babel}
\usepackage[plainpages]{hyperref}
\newcommand{\sh}{\mathop{\mathrm{sh}}\nolimits}
\newcommand{\ch}{\mathop{\mathrm{ch}}\nolimits}
\newcommand{\tnh}{\mathop{\mathrm{th}}\nolimits}
\newcommand{\Isom}{\mathrm{Isom}}

\title{Systole et petites valeurs propres des surfaces hyperboliques} 
\author{Pierre Jammes}
\address{Université Côte d’Azur, CNRS, LJAD, France}
\email{pjammes@unice.fr}
\date{}
\begin{document}
\begin{abstract}
Soit $S$ une surface hyperbolique orientable compacte de caractéristique
d'Euler $\chi$, et $\lambda_k(S)$ la $k$-ième valeur propre non nulle
du laplacien sur $S$. Selon un célèbre résultat d'Otal et Rosas, 
$\lambda_{-\chi}>\frac14$. Dans cet article, nous montrons que
si la systole de $S$ est supérieure à 3,46, alors 
$\lambda_{-\chi-1}>\frac14$. 
Cette inégalité est vraie aussi pour les surfaces hyperbolique 
géométriquement finies sans pointe avec la même hypothèse sur la systole.
\end{abstract}

\begin{altabstract}
Let $S$ be a closed orientable hyperbolic surface with Euler characteristic
$\chi$, and let $\lambda_k(S)$ be the k-th positive eigenvalue for the 
Laplacian on $S$. According to famous result of Otal and Rosas, 
$\lambda_{-\chi}>\frac14$. In this article, we prove that if the
systole of $S$ is greater than 3,46, then $\lambda_{-\chi-1}>\frac14$.
This inequality is also true for geometrically finite orientable 
hyperbolic surfaces without cusps with the same assumption on the systole.
\end{altabstract}
\altkeywords{hyperbolic surfaces, systole, small eigenvalues}


\maketitle

\section{Introduction}
Un problème central de la géométrie spectrale des surfaces hyperboliques
est l'étude des valeurs propres du laplacien contenues dans 
l'intervalle $[0,\frac14]$. Un résultat crucial dans ce domaine, du à
J.-P.~Otal et E.~Rosas est que leur nombre est borné en fonction de la 
topologie de la surface et que sur une surface de caractéristique
d'Euler $\chi$, il y en a
au plus $-\chi$ (rappelons que la caractéristique d'Euler d'une surface
orientable de genre $g$ est $\chi = 2 - 2g$). Si $S$ est une surface 
compacte munie d'une 
métrique hyperbolique, notons $0=\lambda_0(S)<\lambda_1(S) \leq 
\lambda_2(S)\leq\ldots$ les valeurs propres du laplacien sur $S$: 

\begin{theo}[\cite{or09}]\label{intro:th1}
Si $S$ est une surface hyperbolique compacte de caractéristique
d'Euler $\chi$, alors 
$\lambda_{-\chi}(S)>\frac14$.
\end{theo}
\begin{rema}
Dans \cite{or09}, ce théorème est démontré pour toute métrique analytique
de courbure négative (la constante $\frac14$ étant alors remplacée par
le bas du spectre du revêtement universel de $S$). Il a été étendu à
des métriques quelconques dans \cite{bmm16} et à des surfaces non compactes
dans \cite{bmm17}.
\end{rema}

P.~Buser avait donné dans \cite{bu77} des exemples montrant que sur toute 
surface compacte de caractéristique~$\chi$, il existe des métriques 
hyperboliques telles que
$\lambda_{-\chi-1} \leq \frac14$. Le but de cet article est de montrer 
que si on suppose que la systole de $S$ est supérieure à une certaine 
constante explicite et \emph{indépendante de la topologie}, alors de tels 
exemples ne 
peuvent pas exister. Ce résultat s'étend aussi aux surfaces hyperboliques
non compactes ne possédant pas de pointes. Rappelons que la systole 
d'une surface hyperbolique compacte est la longueur minimale de ses 
géodésiques périodiques :

\begin{theo}\label{intro:th2}
 Soit $S$ une surface hyperbolique compacte (ou géométriquement
finie sans pointe) 
orientable de caractéristique d'Euler $\chi$. Si la systole de $S$ est 
supérieure à 3,46, alors $\lambda_{-\chi-1}(S)>\frac14$.
\end{theo}

\begin{rema}
On sait qu'il existe des surfaces hyperboliques compactes de systole 
supérieure à 3,46 entre autres pour les genres 3 à 25 (sauf 8, 12 et 24, 
pour lesquels la question reste ouverte) et tous les 
genres impairs (voir \cite{am16} et les références qui y sont données). 
On conjecture qu'il en existe en tout genre supérieur ou égal à~3.
\end{rema}
\begin{rema}
En genre~2, la systole maximale parmi les métriques hyperboliques est
$2\textrm{arcch}(1+\sqrt2) \simeq 3,06<3,46$, ce maximum étant atteint
par la surface de Bolza. L'hypothèse du théorème~\ref{intro:th2} n'est
donc jamais vérifiée en genre~2, mais on sait que pour la surface
de Bolza on a $\lambda_1\simeq 3,84 >\frac14$ (voir par exemple
\cite{su13}).
\end{rema}
\begin{rema}
Le fait que la constante du théorème~\ref{intro:th2} soit indépendante
de la topologie est assez remarquable et contraste avec d'autres problèmes 
similaires. Par exemple, on n'a pas d'énoncé similaire pour $\lambda_1(S)$:
 une suite de revêtements cycliques d'une surface donnée
donne des exemples de surfaces dont le genre tend vers l'infini, la systole
est uniformément minorée et $\lambda_1\to0$. 
\end{rema}

La démonstration du théorème~\ref{intro:th1} repose sur un fait déjà
remarqué par J.-P.~Otal dans \cite{ot08} : sur une surface hyperbolique, 
un domaine nodal d'une fonction propre de valeur propre $\leq\frac14$
de peut pas être de caractéristique d'Euler~1 ou~0. 
Pour obtenir le 
théorème~\ref{intro:th2}, nous allons montrer que sous son hypothèse 
de minoration de la systole, un tel domaine ne peut pas non 
plus être de caractéristique~-1.
Dans le cas non compact, un argument crucial s'effondre dans le cas
où la surface comprend une pointe (voir la remarque~\ref{critique:rema}
de la section~\ref{critique}).

L'étude de la géométrie des surfaces hyperboliques menée pour démontrer
le théorème~\ref{intro:th2} conduit a un autre résultat qui mérite 
d'être souligné ici. H.~Parlier a montré que pour une métrique 
hyperbolique de systole maximale sur une surface de genre fixée, une
géodésique réalisant la systole ne peut pas être séparante
(\cite{pa12}).
Nous allons montrer une version de ce résultat en genres~2 et~3 qui 
est quantitative en deux sens: d'une part, l'hypothèse n'est pas que
la systole est maximale mais qu'elle dépasse une constante explicite, 
et d'autre part la conclusion est que la longueur d'une géodésique
simple séparante est au moins le double de la systole.

\begin{theo}\label{intro:th3}
Soit $S$ une surface hyperbolique compacte orientable de genre~2 ou~3 et
$\gamma$ une géodésique simple séparante de $S$. Si la systole $s$ de $S$ est 
supérieure à~2,696, alors la longueur de $\gamma$ est supérieure à $2s$.
\end{theo}

La section~\ref{trigo} sera consacré à quelques préliminaires de
géométrie hyperbolique. En corollaire, nous y démontrerons aussi le
théorème~\ref{intro:th3}. Dans la section~\ref{critique}, nous 
donnerons des estimés d'exposants critiques pour les surfaces orientables
non compactes de genre~1. Enfin,  après avoir rappelé les grandes lignes 
de la démonstration du théorème~\ref{intro:th1},
nous démontrerons le théorème~\ref{intro:th2} dans la section~\ref{demo}.

\section{Quelques lemmes de géométrie hyperbolique}\label{trigo}
\subsection{Deux lemmes de trigonométrie hyperbolique}

On considère un hexagone convexe rectangle du plan hyperbolique dont
les cotés consécutifs sont notés $a$, $\gamma$, $b$,
$\alpha$, $c$ et $\beta$ (on utilisera abusivement la même notation pour
un coté et sa longueur). 

Le premier lemme permet d'affirmer que si $a$, $b$ et $c$ sont grands,
$a$ étant le plus grand des trois, alors $\beta$ et $\gamma$ sont petits.
Avec ces hypothèses, $\alpha$ n'est pas forcément petit. On pourra
alors déduire du second lemme que si $\alpha$ est suffisamment grand, 
alors $a\geq b+c$.

\begin{lemm}
Si $a\geq b$ et $a\geq c$, alors
\begin{equation}\label{trigo:ineg1}
\sh \frac\gamma 2\cdot\sh\frac b2\leq\frac12.
\end{equation}

De plus, on a égalité si et seulement si $a=b=c$.
\end{lemm}
\begin{proof}
On part de la relation trigonométrique classique (\cite{bu92}, th.~2.4.1)
\begin{equation}\label{cosinus}
\ch c = \sh a \cdot\sh b \cdot\ch \gamma - \ch a\cdot\ch b
\end{equation}
qui donne :
\begin{equation}
\nonumber
\ch \gamma = \frac{\ch c}{\sh a\cdot \sh b} + \coth a\cdot\coth b.
\end{equation}
Si $b\geq c$, on a la majoration
\begin{equation}\label{trigo:1}
 \frac{\ch c}{\sh a\cdot \sh b} \leq \frac{\ch b}{\sh b^2}
\end{equation}
et si $c\geq b$, on obtient la même inégalité en écrivant
\begin{equation}\label{trigo:2}
\frac{\ch c}{\sh a\cdot \sh b} \leq \frac{\ch c}{\sh c\cdot \sh b}
= \frac{\coth c}{\sh b}\leq \frac{\coth b}{\sh b} = \frac{\ch b}{\sh b^2}.
\end{equation}

On en déduit 
\begin{equation}
\ch \gamma\leq \frac{\ch b}{\sh b^2} + \coth a\cdot\coth b
\leq \frac{\ch b}{\sh b^2} + \coth b^2 = \frac{\ch b + \ch b^2}{\sh b^2}.
\end{equation}

En utilisant les identités entre fonctions hyperboliques, on obtient ensuite
\begin{equation}
2\sh^2\frac\gamma2 = \ch\gamma -1 \leq
\frac{\ch b + \ch^2 b - \sh^2 b}{\sh^2 b} = \frac{\ch b +1}{\sh^2 b}
= \frac{2\ch^2\frac b2}{4\sh^2\frac b2\ch^2\frac b2},
\end{equation}
ce qui se simplifie en $\sh^2\frac\gamma2\leq\frac 1{4\sh^2\frac b2}$,
soit $\sh \frac\gamma 2\cdot\sh\frac b2\leq\frac12$.

Pour établir l'inégalité~(\ref{trigo:1}) on a utilisé le fait que 
$a\geq b$ et $a\geq c$, et en (\ref{trigo:2}) le fait que $a\geq b$ et 
$c\geq b$.
L'égalité dans~(\ref{trigo:ineg1}) va impliquer l'égalité dans (\ref{trigo:1})
ou dans (\ref{trigo:2}) et dans les deux cas on aura l'égalité $a=b=c$.
Réciproquement, on vérifie facilement que si $a=b=c$ on a bien
$\sh \frac\gamma 2\cdot\sh\frac b2 = \frac12$.
\end{proof}

\begin{lemm}\label{trigo:lem2}
 On a l'inégalité $a\geq b+c$ si et seulement si
$\sh^2 \frac\alpha2 \cdot \tnh b \cdot \tnh c \geq 1$.
\end{lemm}


\begin{proof}
Comme le cosinus hyperbolique est une fonction croissante, il suffit de  
montrer que $\sh^2 \frac\alpha2  \cdot \tnh b \cdot \tnh b -1$ et
$\ch a - \ch(b+c)$ sont de même signe.
On part à nouveau de la relation~(\ref{cosinus}) pour écrire :
\begin{eqnarray*}
\ch a - \ch(b+c) & = & \sh b \cdot\sh c \cdot\ch \alpha - \ch b\cdot\ch c\\
& & -(\ch b \cdot \ch c + \sh b \cdot \sh c ) \\
& = &  \sh b \cdot\sh c (\ch \alpha - 1) - 2 \ch b\cdot\ch c\\
& = & 2\sh b \cdot\sh c \cdot \sh^2 \frac\alpha2 - 2 \ch b\cdot\ch c\\
& = & \frac 2{\ch b\cdot\ch c}(\tnh b\cdot\tnh c \cdot \sh^2 \frac\alpha2 -1).
\end{eqnarray*}
 L'inégalité  $\ch a \geq \ch(b+c)$ est donc bien équivalente à  
$\sh^2 \frac\alpha2 \cdot \tnh b\cdot\tnh c -1\geq0$.
\end{proof}

\subsection{Surface de genre (1,1) de grande systole}

Le dernier lemme de cette section établit que si une surface hyperbolique
de genre (1,1) à bord géodésique a une systole suffisamment grande, alors
la longueur du bord est au moins égale au double de la systole. Ce 
résultat permet d'affiner les estimées d'exposant critique de la section
suivante.
Le théorème~\ref{intro:th3} en est aussi une conséquence directe.
\begin{lemm}\label{trigo:lem3}
Soit $S$ une surface hyperbolique de genre $(1,1)$ à bord géodésique. Si
la systole $s$ de $S$ vérifie $s\geq 2,696$, alors la longueur du bord est
supérieure à $2s$.
\end{lemm}

\begin{proof}
On note $\gamma_0$ la géodésique qui borde la surface $S$ et $\gamma_1$ la
plus courte géodésique simple qui découpe $S$ en un pantalon. On découpe
ce pantalon en deux hexagones rectangles isométriques en notant  
$\gamma_2$, $\gamma_3$ et $\gamma_3'$ les arcs qui les bordent, 
comme sur la figure \ref{trigo:pant11}.
\begin{figure}[h]
\begin{center}
\begin{picture}(0,0)%
\includegraphics{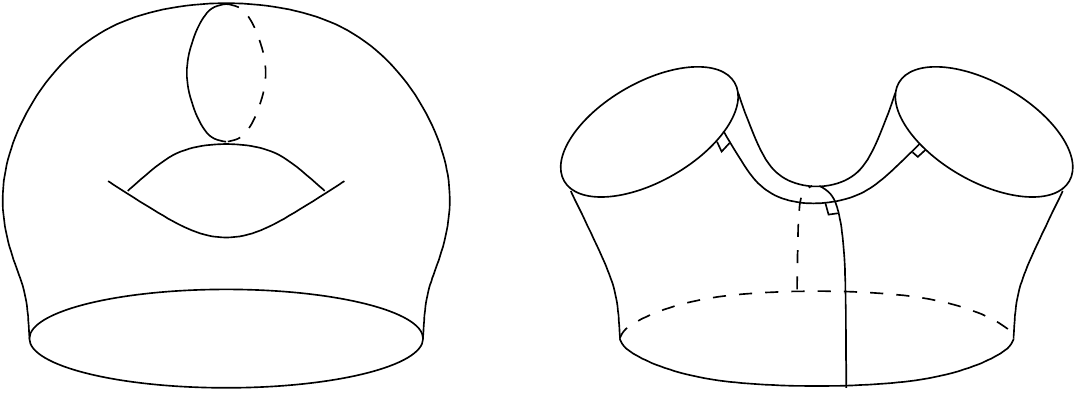}%
\end{picture}%
\setlength{\unitlength}{4144sp}%
\begingroup\makeatletter\ifx\SetFigFont\undefined%
\gdef\SetFigFont#1#2#3#4#5{%
  \reset@font\fontsize{#1}{#2pt}%
  \fontfamily{#3}\fontseries{#4}\fontshape{#5}%
  \selectfont}%
\fi\endgroup%
\begin{picture}(4913,1839)(992,-2605)
\put(2836,-2536){\makebox(0,0)[lb]{\smash{{\SetFigFont{12}{14.4}{\rmdefault}{\mddefault}{\updefault}{\color[rgb]{0,0,0}$\gamma_0$}%
}}}}
\put(5536,-2536){\makebox(0,0)[lb]{\smash{{\SetFigFont{12}{14.4}{\rmdefault}{\mddefault}{\updefault}{\color[rgb]{0,0,0}$\gamma_0$}%
}}}}
\put(1621,-1141){\makebox(0,0)[lb]{\smash{{\SetFigFont{12}{14.4}{\rmdefault}{\mddefault}{\updefault}{\color[rgb]{0,0,0}$\gamma_1$}%
}}}}
\put(3646,-1096){\makebox(0,0)[lb]{\smash{{\SetFigFont{12}{14.4}{\rmdefault}{\mddefault}{\updefault}{\color[rgb]{0,0,0}$\gamma_1$}%
}}}}
\put(4906,-2311){\makebox(0,0)[lb]{\smash{{\SetFigFont{12}{14.4}{\rmdefault}{\mddefault}{\updefault}{\color[rgb]{0,0,0}$\gamma_4$}%
}}}}
\put(5626,-1096){\makebox(0,0)[lb]{\smash{{\SetFigFont{12}{14.4}{\rmdefault}{\mddefault}{\updefault}{\color[rgb]{0,0,0}$\gamma_1$}%
}}}}
\put(5716,-2041){\makebox(0,0)[lb]{\smash{{\SetFigFont{12}{14.4}{\rmdefault}{\mddefault}{\updefault}{\color[rgb]{0,0,0}$\gamma_3'$}%
}}}}
\put(4996,-1726){\makebox(0,0)[lb]{\smash{{\SetFigFont{12}{14.4}{\rmdefault}{\mddefault}{\updefault}{\color[rgb]{0,0,0}$\gamma_2$}%
}}}}
\put(3556,-2041){\makebox(0,0)[lb]{\smash{{\SetFigFont{12}{14.4}{\rmdefault}{\mddefault}{\updefault}{\color[rgb]{0,0,0}$\gamma_3$}%
}}}}
\end{picture}%
\end{center}
\caption{Découpage de la surface \label{trigo:pant11}}
\end{figure}
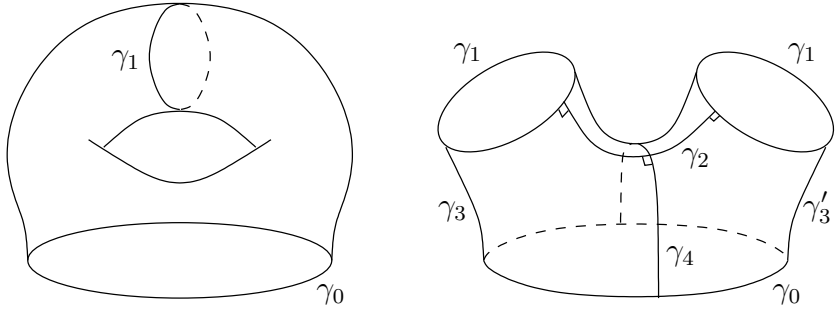

On veut appliquer le lemme~\ref{trigo:lem2} en prenant pour $a$ la moitié
de la géodésique~$\gamma_0$. Les cotés $b$ et $c$ sont alors deux copies
d'un moitié de $\gamma_1$, et $\alpha$ est la géodésique~$\gamma_2$. 
La longueur de $\gamma_1$, qu'on notera $l$, est minorée par $s$. 
Par conséquent $a,b\geq l/2 \geq s/2$.

\begin{figure}[h]
\begin{center}
\begin{picture}(0,0)%
\includegraphics{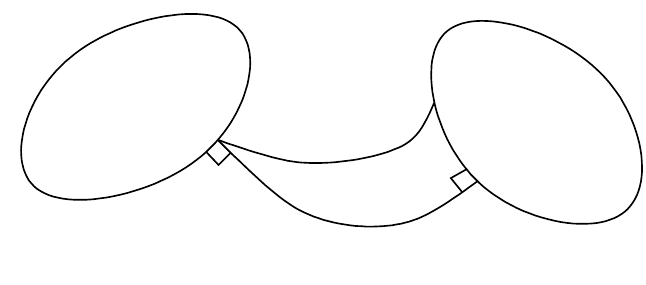}%
\end{picture}%
\setlength{\unitlength}{4144sp}%
\begingroup\makeatletter\ifx\SetFigFont\undefined%
\gdef\SetFigFont#1#2#3#4#5{%
  \reset@font\fontsize{#1}{#2pt}%
  \fontfamily{#3}\fontseries{#4}\fontshape{#5}%
  \selectfont}%
\fi\endgroup%
\begin{picture}(2947,1288)(2081,-2561)
\put(4313,-1999){\makebox(0,0)[lb]{\smash{{\SetFigFont{12}{14.4}{\rmdefault}{\mddefault}{\updefault}{\color[rgb]{0,0,0}$y$}%
}}}}
\put(2891,-1873){\makebox(0,0)[lb]{\smash{{\SetFigFont{12}{14.4}{\rmdefault}{\mddefault}{\updefault}{\color[rgb]{0,0,0}$x$}%
}}}}
\put(2096,-1522){\makebox(0,0)[lb]{\smash{{\SetFigFont{12}{14.4}{\rmdefault}{\mddefault}{\updefault}{\color[rgb]{0,0,0}$\gamma_1$}%
}}}}
\put(3703,-2492){\makebox(0,0)[lb]{\smash{{\SetFigFont{12}{14.4}{\rmdefault}{\mddefault}{\updefault}{\color[rgb]{0,0,0}$\gamma_2$}%
}}}}
\put(4126,-1711){\makebox(0,0)[lb]{\smash{{\SetFigFont{12}{14.4}{\rmdefault}{\mddefault}{\updefault}{\color[rgb]{0,0,0}$x'$}%
}}}}
\put(4676,-1432){\makebox(0,0)[lb]{\smash{{\SetFigFont{12}{14.4}{\rmdefault}{\mddefault}{\updefault}{\color[rgb]{0,0,0}$\gamma_1$}%
}}}}
\end{picture}%
\end{center}
\caption{Minoration de la longueur de $\gamma_2$ \label{trigo:pythagore}}
\end{figure}
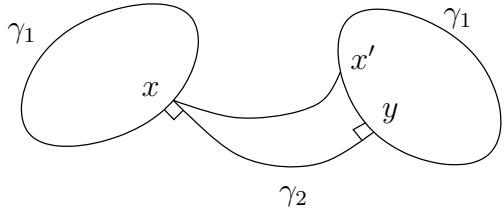

Il reste à minorer~$\alpha$, c'est-à-dire la longueur de la géodésique
$\gamma_2$. Notons $x$ et $y$ les extrémités de $\gamma_2$, et $x'$ le
point situé sur la même copie de $\gamma_1$ que $y$ et qui s'identifie
à $x$ dans $S$ (rappelons que dans $S$, l'arc $\gamma_2$ n'est pas
nécessairement une courbe fermée).

Le triangle $xx'y$ est rectangle en $y$, ce qui permet d'estimer la
longueur de $\gamma_2$ à l'aide du théorème de Pythagore :

\begin{equation}
\cosh (xx') = \cosh (xy) \cdot \cosh(yx').
\end{equation}

Comme $xx' \geq l$ et $yx' \leq l/2$ on en déduit:

\begin{equation}\label{trigo:ineg2}
\cosh \alpha = \cosh (xy) \geq \frac{\cosh l}{\cosh \frac l2}.
\end{equation}

Rappelons que $b$ et $c$ sont minorés par $s/2$. De plus,
$\sinh^2 \frac\alpha{2} = (\cosh \alpha -1)/2$ est minoré grâce à
l'inégalité~(\ref{trigo:ineg2}). On a donc
\begin{equation}
\sh^2 \frac\alpha2 \cdot \tnh b \cdot \tnh c \geq
\frac 12 \cdot \left( \frac{\cosh s}{\cosh \frac s2} - 1 \right)
 \cdot \tnh \frac s2.
\end{equation}
 On peut vérifier que le membre de droite est bien minoré par~1 sous
l'hypothèse $s>2,696$. On peut alors appliquer le lemme~\ref{trigo:lem2}
et conclure que la longueur de $\gamma_0$ est bien minorée par $2s$. 


\end{proof}
\begin{rema}\label{trigo:rem}
Dans la figure~\ref{trigo:pythagore}, on peut minorer $xy$ à l'aide
de l'inégalité triangulaire : $xy\geq xx' - x'y \geq l - l/2 = l/2$.
Cette inégalité, moins précise mais plus pratique que (\ref{trigo:ineg2}),
nous sera utile dans la section suivante.
\end{rema}
\begin{proof}[Démonstration du théorème~~\ref{intro:th3}]
Soit $S$ une surface hyperbolique orientable de genre~2 ou~3 et $\gamma$ 
une géodésique simple séparante de $S$. La géodésique $\gamma$ sépare $S$ 
en deux surfaces orientables 
ayant une seule composante de bord, dont au moins une est de genre~1.
Le lemme~\ref{trigo:lem3} appliqué à cette surface 
de genre~1 assure alors que la longueur de la géodésique $\gamma$
est au moins égale au double de la systole si $s\geq 2,696$.
\end{proof}

\section{Exposants critiques}\label{critique}
Nous allons démontrer dans cette section des majorations d'exposants
critiques pour les groupes fuchsiens convexes cocompacts dont le quotient
est orientable et de caractéristique~-1. En particulier, on cherche une
condition sur la systole pour que l'exposant critique soit strictement
inférieur à $\frac12$. On obtient aussi au passage une majoration
de la série de Poincaré du groupe fuchsien. Rappelons que si $x$ est
un point du plan hyperbolique, la série de Poincaré du groupe $\Gamma$
est définie par
\begin{equation}
P_\Gamma(x,s) = \sum_{g\in\Gamma}e^{-sd(x,g\cdot x)}
\end{equation}
et que son abscisse de convergence ne dépend pas du choix de $x$.

\begin{theo}\label{critique:theo}
Soit $l>0$ et $\Gamma$ un sous-groupe convexe cocompact de 
$\Isom^+(\mathbb H^2)$
tel que la systole de $S=\mathbb H^2/\Gamma$ soit supérieure ou égale à $l$.
\begin{enumerate}
\item Si $S$ est de genre $(0,3)$, alors $\delta_\Gamma \leq \frac{2\ln 2}l$ 
et pour tout point $x$ du c\oe{}ur convexe, on a
$$\displaystyle P_\Gamma(x,s) \leq
\frac{4e^{-s\frac l2}+8 e^{-sl}}{1-e^{-s\frac l2}-2e^{-sl}}$$
 pour $s>\frac{2\ln 2}l$. En particulier, si $l> 4\ln 2$, alors 
 $\delta_\Gamma<\frac12$.
\item Si $S$ est de genre $(1,1)$ et que $l>2,696$, alors $\delta_\Gamma \leq 
\frac{1,73}l $ et il existe un point $x$ du c\oe{}ur convexe tel que
$$\displaystyle P_\Gamma(x,s) \leq
\frac{1 + 3e^{-sl}}{ 1 -  2e^{-sl/2} - e^{-sl}
+ 2 e^{-3sl/2}- 4e^{-2sl}}$$
pour $s>\frac{1,73}l$. En particulier, si $l>3,46$ alors  
$\delta_\Gamma<\frac12$.
\end{enumerate}
\end{theo}

\begin{rema}\label{critique:rema}
Dans le cas où le groupe $\Gamma$ n'est pas convexe cocompact 
(c'est-à-dire qu'il contient des éléments paraboliques, et donc que 
la surface quotient comprend au moins une pointe) un résultat de Beardon
assure que son exposant critique est strictement supérieur à $\frac12$.
\end{rema}

\begin{proof}
Pour majorer cette série et son abscisse de convergence, on va 
chercher une minoration de la distance $d(x,g\cdot x)$.
On fixe un point $x$ dans le (revêtement du) c\oe{}ur convexe. Le point 
de départ de la 
démonstration est de remarquer qu'un segment
géodésique joignant deux points du c\oe{}ur convexe de $S$ est entièrement
contenu dans ce c\oe{}ur. En effet, une géodésique qui en sort
n'y revient jamais. On va donc estimer la
distance $d(x,g\cdot x)$ en se restreignant au revêtement universel du 
c\oe{}ur convexe.

\begin{figure}[h]
\begin{center}
\begin{picture}(0,0)%
\includegraphics{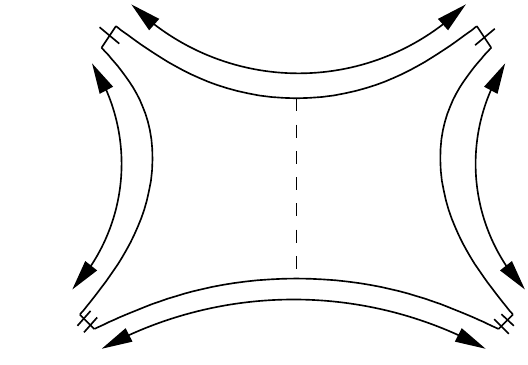}%
\end{picture}%
\setlength{\unitlength}{4144sp}%
\begingroup\makeatletter\ifx\SetFigFont\undefined%
\gdef\SetFigFont#1#2#3#4#5{%
  \reset@font\fontsize{#1}{#2pt}%
  \fontfamily{#3}\fontseries{#4}\fontshape{#5}%
  \selectfont}%
\fi\endgroup%
\begin{picture}(2408,1656)(2146,-1754)
\put(3376,-286){\makebox(0,0)[lb]{\smash{{\SetFigFont{12}{14.4}{\rmdefault}{\mddefault}{\updefault}{\color[rgb]{0,0,0}$\geq l$}%
}}}}
\put(3376,-1681){\makebox(0,0)[lb]{\smash{{\SetFigFont{12}{14.4}{\rmdefault}{\mddefault}{\updefault}{\color[rgb]{0,0,0}$\geq l$}%
}}}}
\put(2161,-916){\makebox(0,0)[lb]{\smash{{\SetFigFont{12}{14.4}{\rmdefault}{\mddefault}{\updefault}{\color[rgb]{0,0,0}$\geq l/2$}%
}}}}
\put(4366,-916){\makebox(0,0)[lb]{\smash{{\SetFigFont{12}{14.4}{\rmdefault}{\mddefault}{\updefault}{\color[rgb]{0,0,0}$\geq l/2$}%
}}}}
\end{picture}%
\end{center}
\caption{Domaine fondamental du c\oe{}ur convexe \label{rel:dom03}}
\end{figure}
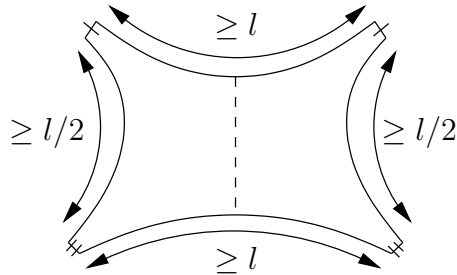

 On peut choisir comme domaine fondamental un octogone rectangle
formé de deux hexagones, comme sur la figure~\ref{rel:dom03}.

\begin{figure}[h]
\begin{center}
\begin{picture}(0,0)%
\includegraphics{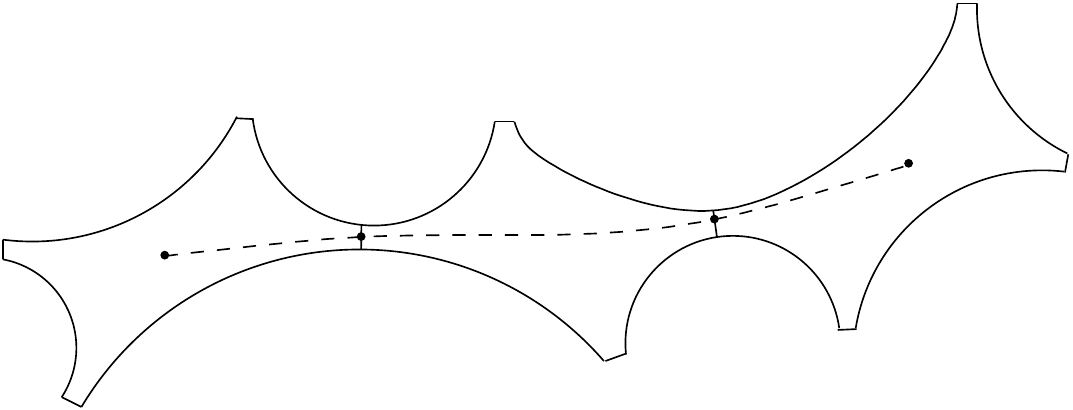}%
\end{picture}%
\setlength{\unitlength}{4144sp}%
\begingroup\makeatletter\ifx\SetFigFont\undefined%
\gdef\SetFigFont#1#2#3#4#5{%
  \reset@font\fontsize{#1}{#2pt}%
  \fontfamily{#3}\fontseries{#4}\fontshape{#5}%
  \selectfont}%
\fi\endgroup%
\begin{picture}(4896,1869)(214,-1648)
\put(658,-1094){\makebox(0,0)[lb]{\smash{{\SetFigFont{12}{14.4}{\rmdefault}{\mddefault}{\updefault}{\color[rgb]{0,0,0}$x$}%
}}}}
\put(4450,-476){\makebox(0,0)[lb]{\smash{{\SetFigFont{12}{14.4}{\rmdefault}{\mddefault}{\updefault}{\color[rgb]{0,0,0}$g\cdot x$}%
}}}}
\put(1785,-1107){\makebox(0,0)[lb]{\smash{{\SetFigFont{12}{14.4}{\rmdefault}{\mddefault}{\updefault}{\color[rgb]{0,0,0}$x_1$}%
}}}}
\put(3397,-1032){\makebox(0,0)[lb]{\smash{{\SetFigFont{12}{14.4}{\rmdefault}{\mddefault}{\updefault}{\color[rgb]{0,0,0}$x_2$}%
}}}}
\end{picture}%
\end{center}
\caption{géodésique dans le revêtement du c\oe{}ur convexe
\label{rel:dom0rev}}
\end{figure}
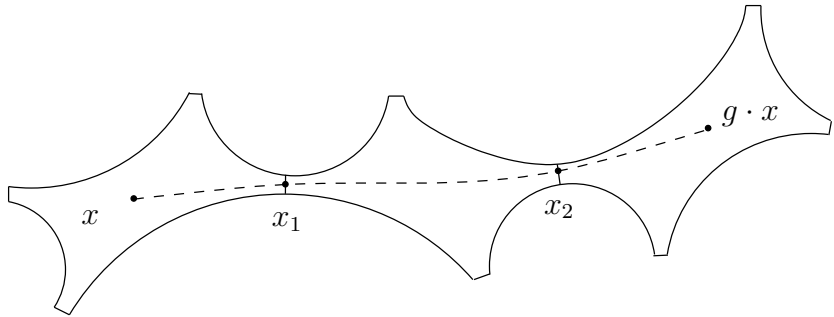

La géodésique de $x$ à $g\cdot x$ traverse plusieurs domaines fondamentaux.
On note $x_1,\ldots,x_n$ les les points d'intersections successifs de
cette géodésique avec les frontières des domaines, comme sur la 
figure~\ref{rel:dom0rev}. On est alors ramené à minorer les distances
$d(x_i, x_{i+1})$. Trois configurations sont possibles, illustrées
par la figure~\ref{rel:dom0l}:
\begin{enumerate}
\item Le segment géodésique $[x_i,x_{i+1}]$ longe un bord du domaine sans
traverser la frontière entre les deux hexagones (ligne pointillée sur 
la figure \ref{rel:dom0l}). On minore sa longueur par une constante
$l_1$ qui sera précisée plus loin.
\item Le segment traverse le domaine en diagonale. On minore sa longueur
par $l_2$.
\item Le segment longe un bord du domaine en changeant d'hexagone et
on minore la longueur par $l_3$. 
\end{enumerate}

\begin{figure}[h]
\begin{center}
\begin{picture}(0,0)%
\includegraphics{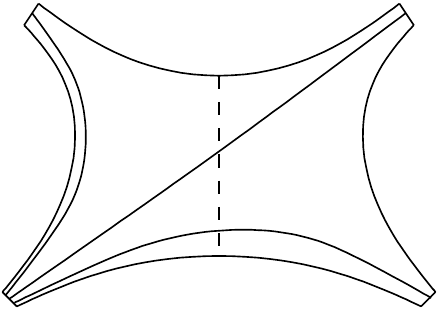}%
\end{picture}%
\setlength{\unitlength}{4144sp}%
\begingroup\makeatletter\ifx\SetFigFont\undefined%
\gdef\SetFigFont#1#2#3#4#5{%
  \reset@font\fontsize{#1}{#2pt}%
  \fontfamily{#3}\fontseries{#4}\fontshape{#5}%
  \selectfont}%
\fi\endgroup%
\begin{picture}(2004,1409)(2499,-1603)
\put(2951,-814){\makebox(0,0)[lb]{\smash{{\SetFigFont{12}{14.4}{\rmdefault}{\mddefault}{\updefault}{\color[rgb]{0,0,0}$\geq l_1$}%
}}}}
\put(3754,-859){\makebox(0,0)[lb]{\smash{{\SetFigFont{12}{14.4}{\rmdefault}{\mddefault}{\updefault}{\color[rgb]{0,0,0}$\geq l_2$}%
}}}}
\put(3709,-1169){\makebox(0,0)[lb]{\smash{{\SetFigFont{12}{14.4}{\rmdefault}{\mddefault}{\updefault}{\color[rgb]{0,0,0}$\geq l_3$}%
}}}}
\end{picture}%
\end{center}
\caption{Segments de géodésique traversant un domaine fondamental
\label{rel:dom0l}}
\end{figure}
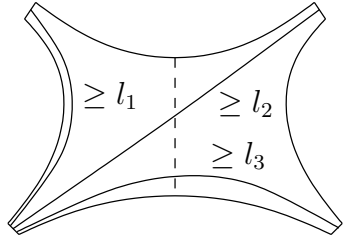

On peut alors écrire la majoration :
\begin{equation}\label{critique:maj0}
d(x,g\cdot x)= d(x,x_1) + \sum_{i=1}^n d(x_{i-1},x_i) + d(x_n,g\cdot x)
\geq\sum_{i=1}^n l_{\alpha_i}, \end{equation}
avec $\alpha_i\in\{1,2,3\}$.

Il est crucial de noter que le domaine fondamental où est situé $g\cdot x$
est entièrement déterminé par la position de $x_1$ et par le mot
$\alpha=\alpha_1\ldots\alpha_n$. L'élément $g\in \Gamma$ lui-même est 
donc déterminé par $\alpha$ et la direction (parmi quatre possibles) 
du segment $[x,x_1]$. On peut donc majorer la série de poincaré
$P_\Gamma(s)$ de la manière suivante:
\begin{equation}\label{critique:maj1}
P_\Gamma(s) \leq 
4\sum_{n=1}^\infty\sum_{|\alpha|=n} e^{-s\sum_{i=1}^n l_{\alpha_i}}
=4\sum_{n=1}^\infty \left(e^{-sl_1}+e^{-sl_2}+e^{-sl_3}\right)^n.
\end{equation}
Cette dernière série converge si $e^{-sl_1}+e^{-sl_2}+e^{-sl_3} < 1$. Compte
tenu de la géométrie du domaine fondamental (cf. Figure\ref{rel:dom03}),
on peut minorer $l_1$ par $l/2$, $l_2$ et $l_3$ par $l$. Il suffit donc
que $2e^{-sl}+e^{-s\frac l2} <1$ pour que la série converge. En posant 
$x=e^{-s\frac l2}$, cette équation devient $2x^2+x-1<0$, ce qui permet 
de voir que la condition  $e^{-s\frac l2}<\frac 12$ est suffisante.
 La série de 
Poincaré converge donc pour tout $s>\frac{2\ln 2}l$ ce qui implique que
$\delta\leq\frac{2\ln 2}l$.

Pour $s>\frac{2\ln 2}l$, la majoration (\ref{critique:maj1}) donne
\begin{equation}
P_\Gamma(s) \leq\frac{4e^{-s\frac l2}+8 e^{-sl}}{1-e^{-s\frac l2}-2e^{-sl}}.
\end{equation}

Dans le cas d'une surface de genre (1,1), on note $\gamma_0$ la géodésique
qui borde le c\oe{}ur convexe et on considère le découpage 
en deux hexagones rectangles définis par $\gamma_0$, $\gamma_1$, 
$\gamma_4$ et $\gamma_3$ et $\gamma'_3$
(cf.  figure~\ref{rel:pant11a}).

\begin{figure}[h]
\begin{center}
\begin{picture}(0,0)%
\includegraphics{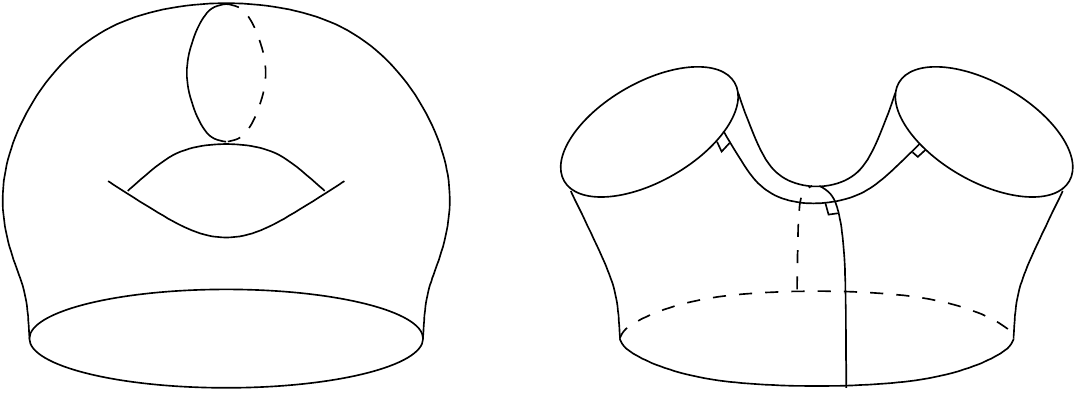}%
\end{picture}%
\setlength{\unitlength}{4144sp}%
\begingroup\makeatletter\ifx\SetFigFont\undefined%
\gdef\SetFigFont#1#2#3#4#5{%
  \reset@font\fontsize{#1}{#2pt}%
  \fontfamily{#3}\fontseries{#4}\fontshape{#5}%
  \selectfont}%
\fi\endgroup%
\begin{picture}(4913,1839)(992,-2605)
\put(2836,-2536){\makebox(0,0)[lb]{\smash{{\SetFigFont{12}{14.4}{\rmdefault}{\mddefault}{\updefault}{\color[rgb]{0,0,0}$\gamma_0$}%
}}}}
\put(5536,-2536){\makebox(0,0)[lb]{\smash{{\SetFigFont{12}{14.4}{\rmdefault}{\mddefault}{\updefault}{\color[rgb]{0,0,0}$\gamma_0$}%
}}}}
\put(1621,-1141){\makebox(0,0)[lb]{\smash{{\SetFigFont{12}{14.4}{\rmdefault}{\mddefault}{\updefault}{\color[rgb]{0,0,0}$\gamma_1$}%
}}}}
\put(3646,-1096){\makebox(0,0)[lb]{\smash{{\SetFigFont{12}{14.4}{\rmdefault}{\mddefault}{\updefault}{\color[rgb]{0,0,0}$\gamma_1$}%
}}}}
\put(4906,-2311){\makebox(0,0)[lb]{\smash{{\SetFigFont{12}{14.4}{\rmdefault}{\mddefault}{\updefault}{\color[rgb]{0,0,0}$\gamma_4$}%
}}}}
\put(5626,-1096){\makebox(0,0)[lb]{\smash{{\SetFigFont{12}{14.4}{\rmdefault}{\mddefault}{\updefault}{\color[rgb]{0,0,0}$\gamma_1$}%
}}}}
\put(5716,-2041){\makebox(0,0)[lb]{\smash{{\SetFigFont{12}{14.4}{\rmdefault}{\mddefault}{\updefault}{\color[rgb]{0,0,0}$\gamma_3'$}%
}}}}
\put(4996,-1726){\makebox(0,0)[lb]{\smash{{\SetFigFont{12}{14.4}{\rmdefault}{\mddefault}{\updefault}{\color[rgb]{0,0,0}$\gamma_2$}%
}}}}
\put(3556,-2041){\makebox(0,0)[lb]{\smash{{\SetFigFont{12}{14.4}{\rmdefault}{\mddefault}{\updefault}{\color[rgb]{0,0,0}$\gamma_3$}%
}}}}
\end{picture}%
\end{center}
\caption{Découpage de la surface \label{rel:pant11a}}
\end{figure}

À la différence du cas précédent, on ne va pas découper le revêtement
universel du c\oe{}ur convexe en domaines fondamentaux mais en domaines 
qui sont des domaines fondamentaux d'un revêtements cyclique
(non compact) du c\oe{}ur de la surface. Un modèle de ces domaines
est représenté sur la figure~\ref{rel:dom11} : il est
obtenu en découpant le c\oe{}ur le long de $\gamma_4$ et en
«~déroulant~» la surface le long de $\gamma_1$. Les arcs géodésiques
$\tilde\gamma_i$ sont des relevés des arcs $\gamma_i$.
\begin{figure}[h]
\begin{center}
\begin{picture}(0,0)%
\includegraphics{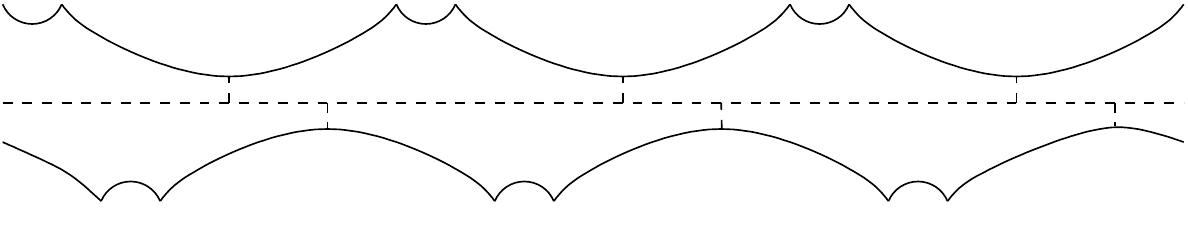}%
\end{picture}%
\setlength{\unitlength}{4144sp}%
\begingroup\makeatletter\ifx\SetFigFont\undefined%
\gdef\SetFigFont#1#2#3#4#5{%
  \reset@font\fontsize{#1}{#2pt}%
  \fontfamily{#3}\fontseries{#4}\fontshape{#5}%
  \selectfont}%
\fi\endgroup%
\begin{picture}(5424,1071)(214,-670)
\put(2566,-601){\makebox(0,0)[lb]{\smash{{\SetFigFont{12}{14.4}{\rmdefault}{\mddefault}{\updefault}{\color[rgb]{0,0,0}$\tilde\gamma_4$}%
}}}}
\put(2971,164){\makebox(0,0)[lb]{\smash{{\SetFigFont{12}{14.4}{\rmdefault}{\mddefault}{\updefault}{\color[rgb]{0,0,0}$\tilde\gamma_3$}%
}}}}
\put(2386,-230){\makebox(0,0)[lb]{\smash{{\SetFigFont{12}{14.4}{\rmdefault}{\mddefault}{\updefault}{\color[rgb]{0,0,0}$\tilde\gamma_1$}%
}}}}
\put(1351,-421){\makebox(0,0)[lb]{\smash{{\SetFigFont{12}{14.4}{\rmdefault}{\mddefault}{\updefault}{\color[rgb]{0,0,0}$\tilde\gamma_0$}%
}}}}
\put(3466,-376){\makebox(0,0)[lb]{\smash{{\SetFigFont{12}{14.4}{\rmdefault}{\mddefault}{\updefault}{\color[rgb]{0,0,0}$\tilde\gamma_3'$}%
}}}}
\end{picture}%
\end{center}
\caption{Revêtement cyclique du domaine fondamental \label{rel:dom11}}
\end{figure}
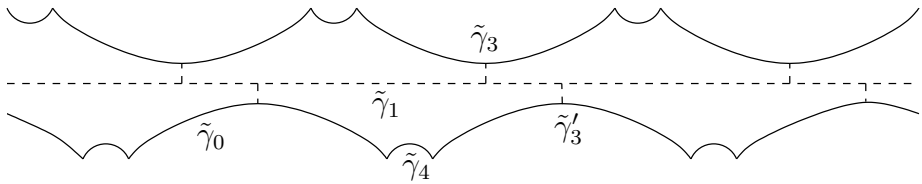

On peut reprendre les majorations~(\ref{critique:maj0}) 
et~(\ref{critique:maj1}) mais avec deux différences. D'une part, 
les $l_{i}$ prennent une infinité de valeurs. D'autre part,
le segment $[x,x_1]$ peut prendre une infinité de directions.
Au lieu de (\ref{critique:maj0}), on va donc écrire
\begin{equation}\label{critique:maj2}
d(x,g\cdot x) \geq d(x,x_1) + \sum_{i=1}^n l_{\alpha_i}, 
\end{equation}
et remplacer le facteur~4 de l'inégalité~(\ref{critique:maj1})
par la somme d'une série portant sur les distances $d(x,x_1)$.

Pour majorer cette somme $\sum e^{-sd(x,x_1)}$, on s'appuie sur 
la figure~\ref{rel:dom11b}. On choisit comme origine $x$ l'extrémité
de $\tilde\gamma_3$ située sur $\tilde\gamma_1$.
On a noté $x'_i$ différentes positions possibles de
$x_1$ en fonction du bord par lequel la géodésique sort du domaine. Pour
chaque $x'_i$, on découpe le segment $[x,x'_i]$ en sous-segments délimités par
les sous-domaines hexagonaux et on minore la longueur de chaque sous-segment
par sa projection orthogonale sur les géodésiques $\tilde\gamma_0$ ou 
$\tilde\gamma_1$.

\begin{figure}[h]
\begin{center}
\begin{picture}(0,0)%
\includegraphics{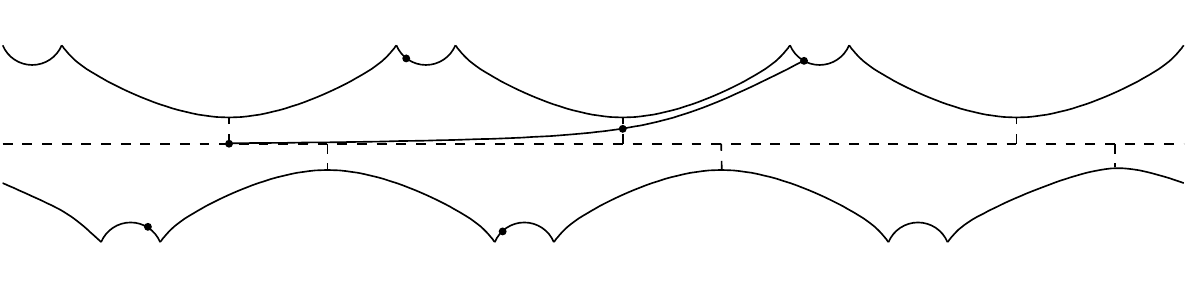}%
\end{picture}%
\setlength{\unitlength}{4144sp}%
\begingroup\makeatletter\ifx\SetFigFont\undefined%
\gdef\SetFigFont#1#2#3#4#5{%
  \reset@font\fontsize{#1}{#2pt}%
  \fontfamily{#3}\fontseries{#4}\fontshape{#5}%
  \selectfont}%
\fi\endgroup%
\begin{picture}(5424,1274)(214,-689)
\put(1184,-210){\makebox(0,0)[lb]{\smash{{\SetFigFont{12}{14.4}{\rmdefault}{\mddefault}{\updefault}{\color[rgb]{0,0,0}$x$}%
}}}}
\put(2503,-625){\makebox(0,0)[lb]{\smash{{\SetFigFont{12}{14.4}{\rmdefault}{\mddefault}{\updefault}{\color[rgb]{0,0,0}$x'_4$}%
}}}}
\put(3872,422){\makebox(0,0)[lb]{\smash{{\SetFigFont{12}{14.4}{\rmdefault}{\mddefault}{\updefault}{\color[rgb]{0,0,0}$x'_2$}%
}}}}
\put(2075,438){\makebox(0,0)[lb]{\smash{{\SetFigFont{12}{14.4}{\rmdefault}{\mddefault}{\updefault}{\color[rgb]{0,0,0}$x'_1$}%
}}}}
\put(2958,-210){\makebox(0,0)[lb]{\smash{{\SetFigFont{12}{14.4}{\rmdefault}{\mddefault}{\updefault}{\color[rgb]{0,0,0}$y$}%
}}}}
\put(755,-610){\makebox(0,0)[lb]{\smash{{\SetFigFont{12}{14.4}{\rmdefault}{\mddefault}{\updefault}{\color[rgb]{0,0,0}$x'_3$}%
}}}}
\end{picture}%
\end{center}
\caption{Longueur de l'arc géodésique initial \label{rel:dom11b}}
\end{figure}
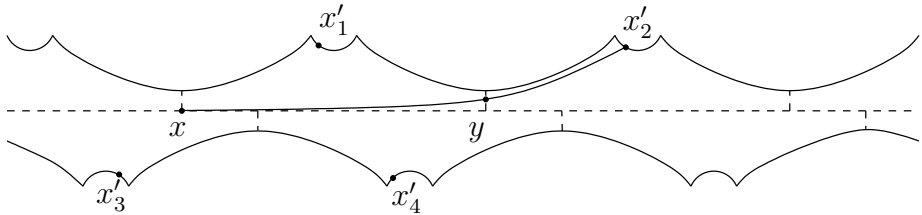

En projetant $[x,x'_1]$ sur $\tilde\gamma_0$, on obtient que 
$d(x,x'_1)\geq l$. 

 Pour minorer la longueur du segment $[x,x'_2]$, on le découpe en deux 
segments $[x,y]$ et $[y,x'_2]$. En projetant le premier sur $\tilde\gamma_1$
et le second sur $\tilde\gamma_0$, on obtient que $d(x,x'_2)\geq 2l$.

En procédant ainsi de suite pour tous les $x'_i$ situés en haut de la
figure~\ref{rel:dom11b}, on obtient une majoration de la somme partielle
$\sum e^{-sd(x,x'_i)}$ par $2\sum_{k=1}^\infty e^{-skl}$.

On procède de manière similaire pour les points $x'_i$ du bas de la figure:
$d(x,x'_3)$ est minoré par 1, $d(x,x'_4)$ est minoré par $l$, etc. On 
obtient une majoration de la somme partielle $\sum e^{-sd(x,x'_i)}$ par 
$1+2\sum_{k=1}^\infty e^{-skl}$.

Globalement, la somme $\sum e^{-sd(x,x_1)}$ est donc majorée par
\begin{equation}
\sum e^{-sd(x,x_1)} \leq 1 + 4\sum_{k=1}^\infty e^{-skl} =
\frac{1 + 3e^{-sl}}{1 - e^{-sl}}.
\end{equation}
On peut donc majorer la série de Poincaré par 

\begin{eqnarray}\label{critique:maj5}
P_\Gamma(s) & \leq &
\frac{1 + 3e^{-sl}}{1 - e^{-sl}}
\sum_{n=1}^\infty\sum_{|\alpha|=n} e^{-s\sum_{i=1}^n l_{\alpha_i}} \nonumber\\
 & \leq & \frac{1 + 3e^{-sl}}{1 - e^{-sl}}
\sum_{n=1}^\infty \left(\sum_i e^{-sl_{\alpha_i}}\right)^n 
\end{eqnarray}
pour les valeurs de $s$ où les séries convergent.

Il reste à majorer la somme de la série $\sum_i e^{-sl_{\alpha_i}}$.
On doit pour cela minorer la longueur des arcs géodésiques qui traverse le
domaine. On va distinguer les géodésiques qui ne coupent pas la
géodésique $\tilde \gamma_1$ (figure~\ref{rel:dom1a}) et celles qui
la coupent (figure~\ref{rel:dom1b}).
\begin{figure}[h]
\begin{center}
\begin{picture}(0,0)%
\includegraphics{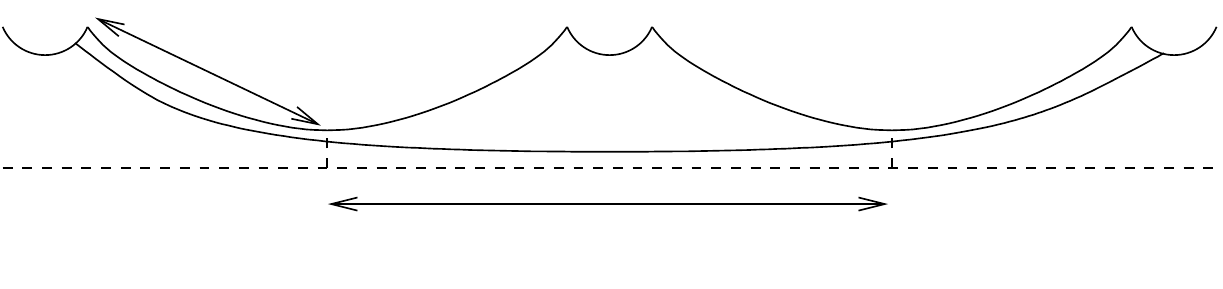}%
\end{picture}%
\setlength{\unitlength}{4144sp}%
\begingroup\makeatletter\ifx\SetFigFont\undefined%
\gdef\SetFigFont#1#2#3#4#5{%
  \reset@font\fontsize{#1}{#2pt}%
  \fontfamily{#3}\fontseries{#4}\fontshape{#5}%
  \selectfont}%
\fi\endgroup%
\begin{picture}(5574,1308)(214,-805)
\put(2701,-736){\makebox(0,0)[lb]{\smash{{\SetFigFont{12}{14.4}{\rmdefault}{\mddefault}{\updefault}{\color[rgb]{0,0,0}$\geq l$}%
}}}}
\put(1036,344){\makebox(0,0)[lb]{\smash{{\SetFigFont{12}{14.4}{\rmdefault}{\mddefault}{\updefault}{\color[rgb]{0,0,0}$\geq l$}%
}}}}
\end{picture}%
\end{center}
\caption{Longueur d'un arc géodésique (premier cas)\label{rel:dom1a}}
\end{figure}
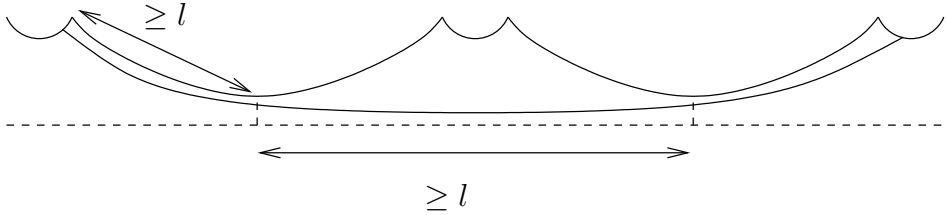
Comme dans le calcul précédent, on découpe l'arc géodésique en segments
travesant un hexagone et on minore la longueur par celle du bord
correspondant de l'hexagone ($\tilde \gamma_1$ ou la moitié de 
$\tilde \gamma_0$, cf. figure~\ref{rel:dom1a}. La somme partielle 
de $\sum_i e^{-sl_{\alpha_i}}$ correspondant à ces arcs est majorée
par 
\begin{equation}
2 (e^{-2sl} + e^{-3sl} + \ldots) = 2 \frac{e^{-2sl}}{1-e^{-sl}}.
\end{equation}
 
\begin{figure}[h]
\begin{center}
\begin{picture}(0,0)%
\includegraphics{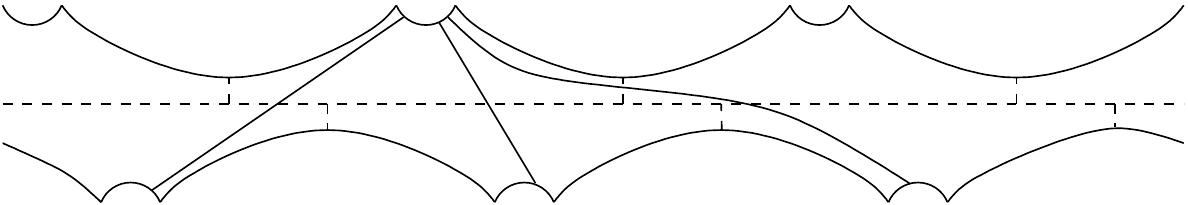}%
\end{picture}%
\setlength{\unitlength}{4144sp}%
\begingroup\makeatletter\ifx\SetFigFont\undefined%
\gdef\SetFigFont#1#2#3#4#5{%
  \reset@font\fontsize{#1}{#2pt}%
  \fontfamily{#3}\fontseries{#4}\fontshape{#5}%
  \selectfont}%
\fi\endgroup%
\begin{picture}(5424,924)(214,-523)
\put(2386,-331){\makebox(0,0)[lb]{\smash{{\SetFigFont{12}{14.4}{\rmdefault}{\mddefault}{\updefault}{\color[rgb]{0,0,0}$a$}%
}}}}
\put(949,-290){\makebox(0,0)[lb]{\smash{{\SetFigFont{12}{14.4}{\rmdefault}{\mddefault}{\updefault}{\color[rgb]{0,0,0}$a'$}%
}}}}
\put(4246,-279){\makebox(0,0)[lb]{\smash{{\SetFigFont{12}{14.4}{\rmdefault}{\mddefault}{\updefault}{\color[rgb]{0,0,0}$b$}%
}}}}
\end{picture}%
\end{center}
\caption{Longueur d'un arc géodésique (second cas)\label{rel:dom1b}}
\end{figure}
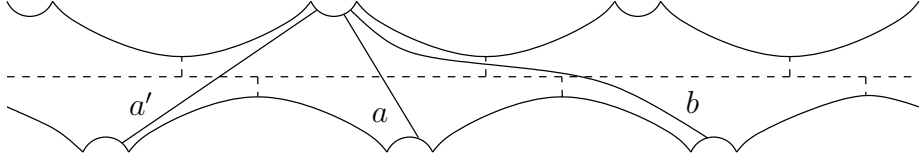

Traitons maintenant le 2\ieme{} cas. Deux géodésiques, $a$ et $a'$ sur 
la figure~\ref{rel:dom1b}, ne passent que par deux hexagones, un de
chaque coté de $\tilde \gamma_1$. Comme la distance entre $\gamma_4$
et $\gamma_1$ est minorée par $l/4$ (cf. remarque~\ref{trigo:rem}),
on peut minorer les longueurs de $a$ et $a'$
par $l/2$. Les autres géodésiques (comme $b$ sur la figure~\ref{rel:dom1b})
se traitent comme dans le premier cas.

La somme $\sum_i e^{-sl_{\alpha_i}}$ est finalement majorée par
\begin{equation}\label{critique:maj6}
2e^{-sl/2} + 4 \frac{e^{-2sl}}{1-e^{-sl}}
\end{equation}
et la série converge si cette expression est majorée par 1. En posant
$x = e^{-sl/2}$, cette condition s'écrit $4x^4 - 2x^3 + x^2 + 2x - 1<0$ 
avec $x>0$). On peut vérifier numériquement que ce polynôme n'a qu'une 
racine positive $x\sim 0,4224$. En particulier, si $sl>1,73$ la série
de Poincaré converge, on a donc $\delta_\Gamma<\frac{1,73}l$.

On obtient une majoration de la série de Poincaré en combinant 
l'inégalité~(\ref{critique:maj5}) et la somme de la série géométrique
de raison (\ref{critique:maj6}). On obtient
\begin{eqnarray}\label{critique:maj7}
P_\Gamma(s) & \leq & \frac{1 + 3e^{-sl}}{1 - e^{-sl}} \cdot
\frac{1}{1 - 2e^{-sl/2} - 4 \frac{e^{-2sl}}{1-e^{-sl}}}\nonumber \\
& \leq & \frac{1 + 3e^{-sl}}{ 1 -  2e^{-sl/2} - e^{-sl} 
+ 2 e^{-3sl/2}- 4e^{-2sl}}.
\end{eqnarray}

\end{proof}

\section{Démonstration du théorème dans le cas des surfaces compactes}
\label{demo}

La démonstration du théorème~\ref{intro:th2} suit celle du
théorème~\ref{intro:th1} et repose sur le lemme de topologie algébrique 
(dit «~de type Borsuk-Ulam~») suivant:

\begin{lemm}[\cite{se02},\cite{or09}]\label{demo:lem1}
Soient $k$ et $n$ deux entiers strictement positifs. Si la sphère $\mathbb S^n$ 
admet une partition en $k$ ensembles $\mathbb S_i$ (non nécessairement 
connexes) tels que:
\begin{itemize}
\item pour tout $i$, $\mathbb S_i$ est invariant par l'involution 
antipodale $\tau$;
\item le revêtement $\mathbb S_i \to \mathbb S_i/\tau$ est trivial; 
\end{itemize}
alors $n\leq k-1$.
\end{lemm}

Ce lemme découle d'un résultat plus général de B.~Sévennec (\cite{se02},
lemme~8).  Il est appliqué dans \cite{or09} à la sphère unité de
l'espace engendré par les fonctions propres associées aux petites
valeurs propres de la surface $S$. Nous allons rappeler les
grandes lignes de la démonstration de \cite{or09} pour ensuite indiquer
comment l'hypothèse de systole minorée permet d'améliorer
le résultat (lemme~\ref{demo:lem4}).

Notons $\mathcal E$ l'espace engendré par les fonctions propres
associées aux valeurs propres $\leq \frac14$ (y compris les fonctions
constantes). Soit $f\in \mathcal E \backslash\{0\}$. Son ensemble nodal
$\gamma(f) = f^{-1}(0)$ est un graphe localement fini (\cite{or09}, 
proposition~5), et donc fini puisqu'on suppose que la 
surface est compacte. Il n'est pas nécessairement connexe et peut contenir
des sommets isolés. Sur chaque composante connexe de $S\backslash 
\gamma(f)$, le signe de $f$ est bien défini.

On note $G(f)$ le graphe obtenu en retirant de l'ensemble nodal les 
composantes contenues dans des disques. Comme chaque composante de
$S\backslash G(f)$ est la réunion d'une composante de 
$S\backslash \gamma(f)$ et de disques disjoints, on peut leur associer
un signe, à savoir celui de $f$ sur la composante de $S\backslash \gamma(f)$.

On note $S^+(f)$ (resp.  $S^-(f)$) la réunion des composantes positives
(resp. négatives) de $S\backslash G(f)$ qui ne sont pas des disques ou
des anneaux. Les ensembles $S^+(f)$ et $S^-(f)$ ne sont pas
nécessairement connexes, peuvent être vides, mais sont incompressibles,
c'est-à-dire que leur groupe fondamental s'injecte dans celui de $S$.

En notant $\chi^+(f)$ (resp.  $\chi^-(f)$) la caractéristique d'Euler
de $S^+(f)$, on a d'une part $\chi^\pm(f)\leq0$ par construction
(avec égalité si et seulement si $S^\pm$ est vide), et d'autre part
 $\chi^+(f) + \chi^-(f) \geq\chi(S)$ du fait de l'incompressibilité.

La démonstration de  \cite{or09} se conclut par deux lemmes. Le
premier exploite le fait que $\mathcal E$ est engendré par les fonctions
propres dont les valeurs propres sont petites :
\begin{lemm}[\cite{or09}, affirmation 6]\label{demo:lem2}
Pour toute fonction non nulle $f\in\mathcal E$, on a 
$\chi^+(f) + \chi^-(f)<0$.
\end{lemm}
L'idée de la démonstration est qu'il existe au moins une composante de
$S\backslash \gamma(f)$ sur laquelle le quotient de Rayleigh de $f$
est inférieur à $\frac14$. Cette composante a nécessairement une
caractéristique d'Euler strictement négative: si elle était égale à~1 ou~0,
on pourrait relever cette composante soit au plan hyperbolique, 
soit à un cylindre hyperbolique. C'est impossible car le relevé de $f$
restreinte à la composante aurait le même quotient de Rayleigh.
Or le plan et le cylindre hyperbolique
ont un spectre égal à $[\frac14;+\infty[$ (le cas où le quotient de
Rayleigh est exactement $\frac14$ est exclu car le relevé de $f$ serait
alors une fonction propre, ce qui est impossible puisqu'il est
nul en dehors d'un compact).

Pour appliquer le lemme~\ref{demo:lem1}, on partitionne la sphère unité
$\mathbb S(\mathcal E)$ en définissant les ensembles $\mathbb S_i$,
$\chi(S)\leq i \leq -1$ par
\begin{equation}
\mathbb S_i = \{f \in \mathcal E\backslash\{0\},\  
\chi^+(f) + \chi^-(f) = i \}.
\end{equation}
Les ensembles $\mathbb S_i$ sont clairement invariants par l'involution
antipodale $\tau$ puisque $\chi^\pm(-f) = \chi^\mp(f)$. Le second lemme
assure que la dernière hypothèse du lemme~\ref{demo:lem1} est
vérifiée:
\begin{lemm}[\cite{or09}, lemme 7]
Pour tout entier $\chi(S)\leq i \leq -1$, le revêtement $\mathbb S_i\to
\mathbb S_i/\tau$ est trivial.
\end{lemm}

Comme les $\mathbb S_i$ sont au nombre de $-\chi(S)$, le 
lemme~\ref{demo:lem1} permet alors d'affirmer que $\dim \mathcal E 
\leq -\chi(S)$.

On obtient le théorème~\ref{intro:th2} en montrant que sous l'hypothèse
de systole minorée, les parties $\mathbb S_i$ ne sont qu'au nombre 
de $-\chi(S)-1$:

\begin{lemm}\label{demo:lem4}
Si la systole de $S$ est supérieure à 3,46, alors $\mathbb S_{-1}$ est
vide.
\end{lemm}
\begin{proof}
La démonstration reprend celle du lemme~\ref{demo:lem2}. Si 
$f\in\mathcal E\backslash\{0\}$ a un quotient de Rayleigh inférieur ou
égal à $\frac14$, il existe une composante $C$ de $S\backslash \gamma(f)$
sur laquelle le quotient de Rayleigh de $f_{|C}$ est aussi $\leq\frac14$.
D'après l'argument de \cite{or09}, cette composante $C$ ne peut
pas se relever au plan ou à un cylindre hyperbolique et sa caractéristique
d'Euler de peut pas être~1 ou~0.

Si la caractérique de $C$ est -1, alors cette composante se relève à une
surface hyperbolique complète $S'$ de caractéristique~-1,
qui est le quotient du plan hyperbolique par le groupe fondamental de $C$
(qu'on peut voir comme un sous-groupe du groupe fondamental de $S$ du
fait de l'incompressibilité). La surface $S'$ est alors un revêtement
riemannien de $S$ et sa systole est au moins égale à celle de $S$.

Si la systole de $S$ (et donc de $S'$) est supérieure à 3,46, un tel 
relevé ne peut pas
exister : en effet, le bas du spectre de $S'$ serait
alors $\leq\frac14$. Elle serait donc le quotient de $\mathbb H^2$ par
un groupe fuchsien d'exposant critique $\leq\frac12$, d'après la théorie de 
Patterson-Sullivan (voir par exemple \cite{bo07}, théorème~14.1). 
Or, on a vu dans la section précédente 
(théorème~\ref{critique:theo}) que c'est
impossible avec une systole plus grande que 3,46. 

 Par conséquent, la caractéristique d'Euler de $C$ est strictement
inférieure à -1, et on a nécessairement $\chi^+(f) + \chi^-(f)<-1$. La
fonction $f$ n'appartient donc pas à $\mathbb S_{-1}$. 
\end{proof}

Le théorème~\ref{intro:th2} en découle immédiatement : les ensembles 
$\mathbb S_i$ qui partitionnent la sphère $\mathbb S(\mathcal E)$ sont
seulement au nombre de $-\chi(S)-1$, donc $\dim\mathcal E \leq -\chi(S)-1$.

\section{Le cas des surfaces non compactes}

La principale difficulté qui apparaît dans le cas non compact est que le
graphe nodal d'une combinaison linéaire de fonctions propres n'est pas 
forcément fini. Comme on est dans un contexte de courbure -1, on 
peut se dispenser des techniques d'ensemble nodal approché utilisées dans 
\cite{bmm16} et \cite{bmm17}, ce qui permet une démonstration plus directe
que dans \cite{bmm17}. L'idée est de se ramener au cas compact au moyen
d'une fonction de coupure: on construit un domaine relativement 
compact $D$ de $S$ difféomorphe à $S$ et un espace de fonctions 
$\mathcal E'$ sur $D$ ayant les mêmes propriétés que les fonctions de 
$\mathcal E$. La démonstration du cas compact peut alors 
s'appliquer à $\mathcal E'$.

\begin{lemm}
Soit $(S,g)$ une surface hyperbolique convexe cocompacte
et $k$ le nombre de valeurs propres du laplacien sur $S$
contenues dans l'intervalle $[0,\frac14[$. Il existe un domaine relativement
compact $D$ de $S$ difféomorphe à $S$ et un espace $\mathcal E'$ de fonctions 
sur $D$ de dimension $k$ tel que
\begin{enumerate}
\item il existe une constante $\lambda\in[0,\frac14[$ telle que le quotient
de Rayleigh de toute fonction de $\mathcal E'$ soit inférieur à $\lambda$;
\item chaque fonction de $\mathcal E'$ vérifie la condition de Dirichlet et
ses lignes nodales sur $\bar D$ sont localement celles d'une fonction 
analytique.
\end{enumerate}
\end{lemm}

Rappelons que la surface est la réunion d'un c\oe{}ur convexe
(qui est une surface hyperbolique compacte à bord géodésique) et de vasques
recollées sur les bords du c\oe{}ur convexe (voir par exemple \cite{bo07},
ch.~2). En vertu d'un théorème de
P.~Lax et R.~S.~Phillips \cite{lp82}, sur une surface hyperbolique d'aire 
infinie, $\frac14$ ne peut pas être valeur propre (cf. \cite{bo07}, ch.~7). 
On se contentera donc 
d'étudier les valeurs propres dans l'intervalle $[0,\frac14[$.

\begin{proof}
 Dans chaque vasque de $S$, on écrit la métrique dans les coordonnées
de Fermi 
\begin{equation}
g = \de r^2 + \cosh ^2 r \de t^2,
\end{equation}
où la courbe $r=0$ est la géodésique bordant la vasque.

Pour tout $\varepsilon>0$, on définit une fonction de coupure 
$\eta_\varepsilon$ par
\begin{equation}
\left\{
\begin{array}{ll}
\eta_\varepsilon = 1 & \textrm{si } r \leq \frac1\varepsilon \\
\eta_\varepsilon = 2 - \varepsilon r & \textrm{si } 
\frac1\varepsilon < r < \frac2\varepsilon \\
\eta_\varepsilon = 0 & \textrm{si } r \geq \frac2\varepsilon \\
\end{array}
\right.
\end{equation}
et on prolonge $\eta_\varepsilon$ par 1 en dehors des vasques. Son
gradient est nul partout sauf pour 
$r\in]\frac1\varepsilon, \frac2\varepsilon[$ où 
$|\de\eta_\varepsilon| = \varepsilon$, et pour $r=\frac1\varepsilon,\ 
\frac2\varepsilon$ où il n'est pas défini.

 Pour chaque $\varepsilon>0$, on définit le domaine $D_\varepsilon$ de $S$
comme la réunion du  c\oe{}ur convexe et des domaines définis par
$r<\frac2\varepsilon$ dans chaque vasque. On veut montrer qu'il existe
un $\varepsilon$ tel que la conclusion du lemme soit vérifiée pour 
le domaine $D_\varepsilon$.

Soit $\lambda_S$ la plus grande des valeurs propres du laplacien
sur $S$ contenue dans l'intervalle $[0,\frac14[$.


 Pour tout $f\in \mathbb S$, l'intégrale $\int_{r>\frac1\varepsilon}f^2$ tend
vers~0 en décroissant quand $\varepsilon\to0$, de même que
$\int_{r>\frac1\varepsilon}|\de f|^2$. D'après le théorème de Dini, 
la famille d'applications $T_\varepsilon:\mathbb S\to\R$
définie par $T_\varepsilon(f)=\int_{r>\frac1\varepsilon}f^2+
\int_{r>\frac1\varepsilon}|\de f|^2$ converge donc
uniformément vers la fonction nulle quand $\varepsilon\to0$.

Pour tout $f\in \mathbb S$, on pose $f_\varepsilon=\eta_\varepsilon f$. Comme
$\|f-f_\varepsilon\|^2\leq T_\varepsilon(f)$, $\|f_\varepsilon\|$
converge uniformément vers 1 sur $\mathbb S$, quand $\varepsilon\to0$.

De plus, comme $\de f_\varepsilon=f\de\eta_\varepsilon+\eta_\varepsilon\de f$,
on a,
\begin{eqnarray*}
\|\de f_\varepsilon-\de f\|&\leq& \|f\de\eta_\varepsilon\|+
\|(1-\eta_\varepsilon)\de f\|\\
&\leq& \varepsilon+\sqrt{T_\varepsilon(f)}.
\end{eqnarray*}
Donc $\|\de f_\varepsilon\|$ converge uniformément sur $\mathbb S$ vers
$\|\de f\|$. On en déduit que le quotient de Rayleigh $R(f_\varepsilon)$
de $f_\varepsilon$ converge uniformément vers celui de $f$ quand
$\varepsilon\to0$. Comme le quotient de Rayleigh sur $\mathbb S$ est majoré
par $\lambda_S<\frac14$, on peut trouver $\varepsilon$ suffisamment
petit et $\lambda\in]\lambda_S,\frac14[$ tels que $R(f_\varepsilon)
<\lambda$ pour tout $f$.

Un tel $\varepsilon$ étant fixé, on définit $D=D_\varepsilon$
$\mathcal E'=\{(\eta_\varepsilon f)_{|D},\ f\in \mathcal E\}$. 
Par construction, le quotient de Rayleigh est majoré par $\lambda$ sur 
$\mathcal E'$. 
Comme les fonctions $f\in \mathcal E$ et $r$ sont analytiques, les lignes 
nodales de $f_\varepsilon$ sont bien localement celles d'une fonction 
analytique : celles de $f$ à l'intérieur de $D$ et celles de 
$\eta_\varepsilon f$ au voisinage du bord de $\bar D$.
\end{proof}

\providecommand{\bysame}{\leavevmode ---\ }
\providecommand{\og}{``}
\providecommand{\fg}{''}
\providecommand{\smfandname}{\&}
\providecommand{\smfedsname}{\'eds.}
\providecommand{\smfedname}{\'ed.}
\providecommand{\smfmastersthesisname}{M\'emoire}
\providecommand{\smfphdthesisname}{Th\`ese}

\end{document}